\documentclass[a4paper,leqno,11pt]{article}

\usepackage{
amsmath,
amssymb,
amsthm,
amsfonts
}

\usepackage[hidelinks]{hyperref}

\usepackage{geometry}
\newcommand\LL{{\mathcal L}}
\newcommand\R{\mathbb{R}}
\newcommand\ep{\varepsilon}
\newcommand\eps{\varepsilon}
\newcommand\dt{\,dt}
\newcommand\ds{\,ds}
\newcommand\dx{\,dx}

\newtheorem{theorem}{Theorem}[section]
\newtheorem{proposition}[theorem]{Proposition}
\newtheorem{lemma}[theorem]{Lemma}
\newtheorem{corollary}[theorem]{Corollary}

\theoremstyle{definition}
\newtheorem{definition}[theorem]{Definition}

\newtheorem{assumption}[theorem]{Assumption}

\theoremstyle{remark}
\newtheorem{remark}[theorem]{Remark}

\title{De Giorgi's approach to hyperbolic Cauchy problems:\\the case of nonhomogeneous equations}

\author{Lorenzo Tentarelli$^\dagger$ and Paolo Tilli$^\ddagger$
\\ \ \\
{\small $^\dagger$Dipartimento di Matematica e Applicazioni ``R. Caccioppoli'' } \\
{\small Universit\`a degli Studi di Napoli ``Federico II'' } \\
{\small Via Cintia, Monte S. Angelo, I-80126 Napoli, Italy} \\
{\small \texttt{lorenzo.tentarelli@unina.it}}\\ \ \\
{\small  $^\ddagger$Dipartimento di Scienze Matematiche ``G.L. Lagrange'' } \\
{\small Politecnico di Torino } \\
{\small Corso Duca degli Abruzzi, 24, 10129 Torino, Italy} \\
{\small \texttt{paolo.tilli@polito.it}}
}

\begin{document}

\maketitle

\begin{abstract}
In this paper we discuss an extension
of some results obtained by E. Serra and P. Tilli, in \cite{ST1,ST2}, concerning an original conjecture by E. De Giorgi (\cite{degiorgi1, degiorgi2}) on a \emph{purely minimization} approach to the Cauchy problem for the defocusing nonlinear \emph{wave equation}. Precisely, we show how to extend the techniques developed by Serra and Tilli for homogeneous hyperbolic nonlinear PDEs to the \emph{nonhomogeneous} case, thus proving that the idea of De Giorgi yields in fact an effective approach to investigate general hyperbolic equations.
\end{abstract}

\noindent{\small AMS Subject Classification: 35L70, 35L71, 35L75, 35L76, 35L90, 49J45.}
\smallskip

\noindent{\small Keywords: nonlinear hyperbolic equations, mimimization, nonhomogeneous PDEs, De Giorgi conjecture.}


\section{Introduction}

In this paper we present an extension, to the case of nonhomogeneous equations,
of some recent results obtained in \cite{ST1,ST2} on
a \emph{minimization approach} to hyperbolic Cauchy problems. This approach
was originally
suggested by E. De~Giorgi  through a conjecture (\cite{degiorgi1,degiorgi2}),
essentially proved in \cite{ST1} and then extended to an abstract setting in
\cite{ST2} (see also \cite{serra,stefanelli,tentarelli} and references therein).

More precisely, we introduce a suitable variant of this method in order to investigate
hyperbolic PDEs having the formal structure of
\begin{equation}
 \label{eq-wave}
 w''(t,x)=-\nabla\mathcal{W}\big(w(t,\cdot\,)\big)(x)+f(t,x),\qquad(t,x)\in\R^+\times\R^n,
\end{equation}
with two prescribed initial conditions
\begin{equation}
 \label{eq-ic}
 w(0,x)=w_0(x),\quad\quad w'(0,x)=w_1(x),\qquad x\in\R^n.
\end{equation}
Here, as in \cite{ST2}, $\nabla\mathcal{W}$ is the G\^ateaux derivative of a functional $\mathcal{W}:\mathrm{W}\to[0,\infty)$ ($\mathrm{W}$ is some Banach space of functions in $\R^n$, typically a Sobolev space), the main novely being that
 we allow
for a function
$f(t,x)$ in \eqref{eq-wave}, that acts as a forcing term (a source) in the resulting PDE.

The idea behind De Giorgi's approach is to obtain solutions of hyperbolic Cauchy problems as limits
(when $\ep\downarrow0$)
of the \emph{minimizers} $w_\eps$ of a sequence of suitable functionals $F_\eps$
of the Calculus of Variations, defined as integrals in space-time of a suitable Lagrangian with
an exponential weight.
 De~Giorgi's conjecture, in its original formulation \cite{degiorgi1}, concerns the
 defocusing NLW equation
\begin{equation}
 \label{eq-NLW}
 w''=\Delta w-|w|^{p-2}\,w\qquad(p\geq2),
\end{equation}
which falls within the general scheme \eqref{eq-wave} if we let
\[
 \mathcal{W}(v)=\int_{\R^n}\left(\frac{1}{2}|\nabla v|^2+\frac{1}{p}|v|^p\right)\dx\qquad\text{and}\qquad f\equiv0.
\]
If $w_\eps$ denotes the minimizer of the \emph{convex} functional in space-time
\begin{equation}
\label{defFh}
F_{\ep}^h(w):=
\int_0^{\infty}e^{-t/\ep}\,\left(\frac{\ep^2}{2}\int_{\R^n}|w''(t,x)|^2\dx+\mathcal{W}\big(w(t,\cdot\,)\big)\right)\dt
\end{equation}
subject to the \emph{boundary} conditions \eqref{eq-ic},
De~Giorgi conjectured that $w_\eps\to w$, where $w$ solves \eqref{eq-NLW} and satisfies \eqref{eq-ic}, now meant as \emph{initial} conditions of the Cauchy problem
(for more details see \cite{degiorgi1,nirenberg2,ST1}).
This conjecture was essentially proved
 in \cite{ST1} (see also \cite{stefanelli}),
and then generalized in \cite{ST2} with an abstract version of the result,
which shows that the NLW equation \eqref{eq-NLW} can be replaced with
the abstract equation
\eqref{eq-wave} for quite general functionals $\mathcal{W}$,
but still in the \emph{homogeneous} case where $f\equiv0$.

Of course, when a nontrivial source $f(t,x)$ is present
in \eqref{eq-wave}, the functional $F_{\ep}^h$ defined in \eqref{defFh}
(being independent of $f$) is no longer appropriate: instead of $F_{\ep}^h$,
a natural choice is to minimize, subject to the boundary conditions \eqref{eq-ic},
the functional
\begin{equation}
 \label{eq-funz}
 F_{\ep}(w):=F_{\ep}^h(w)-F_{\ep}^s(w)
\end{equation}
where $F_{\ep}^s$ is the \emph{linear} functional
\begin{equation}
 \label{eq-funzsor}
 F_{\ep}^s(w):=\int_0^{\infty}\!\int_{\R^n}e^{-t/\ep}\,f_{\ep}(t,x)w(t,x)\dx\dt
\end{equation}
and $f_{\ep}(t,x)$ is a suitable \emph{approximation} of $f(t,x)$.
Intuitively, this can be justified  by the following heuristic argument:
if  $w_{\ep}$ is a minimizer of $F_{\ep}$ subject to \eqref{eq-ic},
by elementary computations one can check that
 the Euler-Lagrange equations for $F_\eps$
 reduce to
\begin{equation}
\label{eul}
 \ep^2\,w_{\ep}''''(t,x)-2\ep w_{\ep}'''(t,x)+w_{\ep}''(t,x)=-\nabla\mathcal{W}\big(w_{\ep}(t,\cdot\,)\big)(x)+f_{\ep}(t,x).
\end{equation}
 Now the connection with \eqref{eq-wave} is clear: when $\ep\downarrow0$,
 assuming that $f_{\ep}\to f$ and $w_{\ep}\to w$, one formally obtains \eqref{eq-wave}
(coupled with \eqref{eq-ic})
in the limit (of course choosing $f_\eps=f$ would seem most natural, but unfortunately
this is not possible, as we shall explain later).

In this paper we show that this procedure can be carried out successfully, under the
sole assumption that $f\in L^2_{\text{loc}}([0,\infty);L^2)$, and under very general
assumptions on the functional $\mathcal{W}$ (namely Assumption~\ref{ass-W} and
\eqref{eq-Wass}, as in \cite{ST2}). Our results are summarized in Theorem~\ref{result},
which is a natural development of the research program initiated in \cite{ST1,ST2}
(in fact, letting $f\equiv 0$ in Theorem~\ref{result}, one obtains all the results
of \cite{ST2} as a particular case). In order to illustrate the wide variety of
nonhomogeneous equations covered by Theorem~\ref{result} we refer to Section~\ref{sec-examples}
which, being independent of the technical parts of the paper, can serve as a supplement to this introduction.

We wish to stress that this is not just a technical extension
of the results in \cite{ST2}. Indeed, in \cite{ST1,ST2}
the main ingredient
 to obtain estimates on the minimizers $w_\eps$
is a control
(uniform in $\eps$) of the quantity
\begin{equation}
\label{defea}
 \mathcal{E}_{\ep}(t):=\frac{1}{2}\int_{\R^n}|w_{\ep}'(t,x)|^2\dx
+\eps^{-2}\int_0^{\infty}s \,e^{-s/\ep} \,\mathcal{W}\bigl(w_{\ep}(t+s)\bigr)\ds,
\end{equation}
the so called \emph{approximate energy}, to be compared (in view of $w_\eps\to w$) to
\[
 \mathcal{E}(t):=\frac{1}{2}\int_{\R^n}|w'(t,x)|^2\dx
+\mathcal{W}\bigl(w(t)\bigr),
\]
the natural energy for a solution of \eqref{eq-wave}, which is formally preserved
when $f\equiv 0$. Now, contrary to $\mathcal{E}(t)$ which
depends only on $w'(t)$ and ${\mathcal W}\bigl(w(t)\bigr)$,
we see that the potential term in \eqref{defea} (the
integral involving $\mathcal W$)
depends on the values of $\mathcal{W}\bigl(w_\ep(\tau)\bigr)$ for all $\tau\geq t$: following \cite{tao},
we say that this term is ``acausal''.

This acausality is deep-seated: since \eqref{eul} is
of the fourth order in $t$, prescribing \emph{two} initial conditions as in \eqref{eq-ic}
is not enough to uniquely determine the evolution of $w_\eps(t)$. On the other hand, $w_\eps$
is obtained as a \emph{minimizer} of $F_\eps$ subject to \eqref{eq-ic}, and
 the minimization procedure certainly \emph{selects}, among the infinitely many solutions of
\eqref{eul}\&\eqref{eq-ic}, \emph{one} with special features (such as the finiteness of
$F_\eps(w_\eps)$, which is trivial for a minimizer, but does not follow from \eqref{eul}\&\eqref{eq-ic}). Thus, the fact that the global-in-time behaviour of $w_\eps(s)$
is relevant for the approximate energy $\mathcal{E}_\ep(t)$ is not surprising. Note, however,
that the function $\eps^{-2}s e^{-s}$ in \eqref{defea} is a \emph{probability measure}
on $s>0$, which concentrates at $s=0$ when $\ep\downarrow0$: therefore, the second integral in
\eqref{defea} is just an (acausal) average of $\mathcal{W}\big(w_\eps(\tau)\big)$ for $\tau\geq t$,
which concentrates around $\tau=t$ for small $\eps$ (so that, heuristically,
acausality becomes negligible when $\ep\downarrow0$, as long as smoothness is assumed).

Now, in the nonhomogeneous case where $f\not\equiv 0$, the approximate
energy $\mathcal{E}_\ep(t)$ (as defined in \eqref{defea}) is again the natural object to estimate.
But we see from \eqref{eq-funz} and \eqref{eq-funzsor} that the presence of $f_\eps$
may strongly influence the behaviour of $w_\eps(t)$, possibly in a \emph{global} (hence
also acausal) way: and this is in contrast with the limit problem \eqref{eq-wave}\&\eqref{eq-ic},
where the solution $w(T)$ depends only on $f(t)$ restricted to  $t\in [0,T]$, in
a strictly causal way.

This calls for some new ideas, in addition to those introduced in \cite{ST1,ST2}, in order
to obtain strong enough estimates on $w_\eps$, pass to the limit in \eqref{eul},
and obtain  sharp energy estimates as in \eqref{eq-enineq}. Therefore, in our proofs,
we shall mainly focus on these new aspects, referring to \cite{ST2} for those
lemmas or computations which do not require significant changes.

In \cite{ST2}, where $f\equiv 0$, it is proved that $\mathcal{E}_\ep'(t)\leq 0$,
so that $\mathcal{E}_\ep(t)\leq \mathcal{E}_\ep(0)\leq\mathcal{E}(0)+o(1)$, and this is the key
to all the subsequent estimates for $w_\eps(t)$, uniform in $\eps$ and $t$. Here, instead,
the presence of the forcing term $f$ prevents any apriori monotonicity, and every
bound for $\mathcal{E}_\ep(t)$ will depend on $f$ itself. In fact,
$\mathcal{E}_\ep'(t)$ depends on $f$ (more
precisely on $f_\eps$) in a nonlocal, acausal way (see Section~\ref{sec-est_en}), and this
requires new strategies and a careful analysis based (among the other things) on the tools introduced in Section~\ref{sec-average}.
Indeed, we can obtain estimates only over bounded time intervals and up to some residual terms, which however can be proved to vanish in the limit when $\ep\downarrow0$.

In the light of these considerations, it appears that letting $f_\eps=f$ in \eqref{eq-funzsor}
(though formally correct) is not appropriate, and some
nontrivial approximation
$f_\eps\to f$ is therefore mandatory. Moreover, we wish to work with $f\in L^2_{\text{loc}}([0,\infty);L^2)$
(which is the natural assumption if one seeks solutions of \eqref{eq-wave} with \emph{finite energy}
 -- see e.g.~\cite{cherrier,lions2}),
while the integral in \eqref{eq-funz}, in order to be defined, requires some
restriction on the growth of $\Vert f_\ep(t)\Vert_{L^2}$.
In any case, we stress that the choice of the sequence $f_{\ep}$ is crucial in the
detection and the control of the residual terms in our estimates (Lemma \ref{lem-fapp}, Corollary \ref{cor-fond} and Remark \ref{rem-connections}).

\medskip

The full strength of Theorem~\ref{result} is obtained under the structural
assumption \eqref{eq-Wass}, which forces the evolution equation
\eqref{eq-wave} to be semilinear (albeit of
arbitrary order in space, including waves equations with the fractional Laplacian -- see
Section~\ref{sec-examples}).

It should be pointed out, however, that assumption
\eqref{eq-Wass} is required only  in item (e) of Theorem~\ref{result}
(the passage
to the limit in \eqref{eul} to obtain \eqref{eq-wave}):
all the other claims of the theorem (items (a)--(d), including estimates and convergence
to a function that satisfies the energy inequality)
are valid in the much wider setting
of Assumption~\ref{ass-W},
which is typically satisfied by any reasonable functional of the Calculus of Variations
(not necessarily convex, and possibly nonlocal). As shown in Section~\ref{sec-examples},
this broad framework includes wave equations with the $p$-Laplacian
 such as \eqref{eq-quasilinear} and nonlocal
evolutions like the Kirchhoff equation \eqref{kirch}, for which the existence of
global weak solutions is an open problem: the validity of items (a)--(d) of Theorem~\ref{result}
for these equations suggests a possible new strategy in this direction, since no
counterexample is known to the claim of item (e), for which \eqref{eq-Wass} is just
a sufficient condition.

Of course, in several concrete examples where \eqref{eq-Wass} is satisfied
(e.g. the NLW equation \eqref{dNLW}) the existence of global weak solutions
provided by Theorem~\ref{result}
is not new
(for an overview of other techniques we refer the reader to \cite{cherrier,lions4,segal,shastru1,shastru2,shastru3,strauss,struwe,tao} and references therein).
However, we stress that the
variety of different examples of equations that
can be treated by this unifying approach is remarkable,
and we believe that this variational technique  would deserve further investigations.

Finally, we recall that suitable variants of this variational approach
to evolutions problems
have recently been developed to study  other kind of equations: we refer the reader to
 \cite{akagi,bogelein,melchionna} for applications to parabolic equations, and to \cite{liero}
(and references therein) for the application to ODE systems.

\bigskip
\bigskip

\noindent\textbf{Remark on Notation.} If $g=g(t,x)$,
we write $g(t)$ or equivalently $g(t,\cdot\,)$ to denote the function of $x$
that is obtained fixing $t$. We also write $g'$, $g''$ etc. to denote partial derivatives with respect to $t$, while differential operators like $\nabla$, $\Delta$ etc.
are referred to the space variables only.
Concerning function spaces, we agree that $L^p=L^p(\R^n)$, $H^m=H^m(\R^n)$ etc.,
the domain $\R^n$ being understood.
Finally, $\langle\cdot\,,\cdot\,\rangle$ denotes  a duality pairing (usually clear from the
context), while $(\cdot\,,\cdot\,)_H$ denotes the inner product in a Hilbert space $H$.

\bigskip
\bigskip
\noindent\textbf{Acknowledgements}

\medskip
\noindent L.T. acknowledges the support of MIUR through the FIR grant 2013 ``Condensed Matter in Mathematical Physics (Cond-Math)'' (code RBFR13WAET).


\section{Functional setting and main results}

The abstract equation \eqref{eq-wave} and the functional \eqref{defFh} are defined
in terms of the abstract functional $\mathcal{W}$.
In order to develop our approach, the properties  that $\mathcal W$ must safisfy
are the same as in \cite{ST2}, and can be summarized as follows.

\begin{assumption}
 \label{ass-W}
 The functional $\mathcal{W}:L^2\to[0,\infty]$ is lower semicontinuous in the weak topology of $L^2$. Moreover, we assume that its domain, i.e. the set of functions
 \begin{equation}
  \label{eq-domW}
  \mathrm{W}:=\{v\in L^2:\mathcal{W}(v)<\infty\},
 \end{equation}
 is a Banach space such that
 \begin{equation}
  \label{eq-dom_emb}
  C_0^{\infty}\hookrightarrow\mathrm{W}\hookrightarrow L^2\qquad\mbox{(dense embeddings).}
 \end{equation}
 Finally, $\mathcal{W}$ is G\^ateaux differentiable on $\mathrm{W}$ and its derivative $\nabla\mathcal{W}:\mathrm{W}\to\mathrm{W}'$ satisfies
 \begin{equation}
  \label{eq-diffW}
  \|\nabla\mathcal{W}(v)\|_{\mathrm{W}'}\leq C\big(1+\mathcal{W}(v)^{\theta}\big),\qquad\forall v\in\mathrm{W},
 \end{equation}
 for suitable constants $C\geq0$ and $\theta\in(0,1)$.\qed
\end{assumption}

\begin{remark}
This assumption (in particular, inequality \eqref{eq-diffW}) is typical of Dirichlet-type functionals like $\mathcal{W}(v)=\Vert \nabla^k v\Vert_{L^p}^p$ with $p>1$
 (in this case  $W$ is a suitable Sobolev space).  We refer to Section \ref{sec-examples}
for some examples. Here we just point out
that Assumption \ref{ass-W} is \emph{additively stable}, i.e. if two functionals satisfy Assumption \ref{ass-W}, then so does their sum (for further remarks on this assumption, see \cite{ST2}).
\end{remark}

\begin{theorem}
 \label{result}
 Let $\mathcal{W}$ be a functional satisfying Assumption \ref{ass-W} and $w_0,\,w_1\in\mathrm{W}$. Let also $f\in L_{\text{loc}}^2([0,\infty),L^2)$. Then, there exists a sequence $(f_\ep)$, converging to $f$ in $L_{\text{loc}}^2([0,\infty),L^2)$, such that:
 \begin{itemize}
  \item[(a)] \textbf{Minimizers}. For every $\ep\in(0,1)$, the functional $F_{\ep}$ defined by \eqref{eq-funz} has a minimizers $w_{\ep}$, among all functions in $H^2_{\text{loc}}([0,\infty);L^2)$ that satisfy \eqref{eq-ic}.
  \item[(b)] \textbf{Estimates}. For every $T>0,\tau\geq0$, there exist constants $C_T,\,C_{\tau,T}$
  independent of $\ep$ such that
  \begin{equation}
  \label{eq-apruno}
  \sup_{t\in [0,T]}\int_{\R^n}\big(\,|w_{\ep}'(t,x)|^2+|w_{\ep}(s,x)|^2\,\big)\dx\leq C_{T},
 \end{equation}
 \begin{equation}
  \label{eq-aprdue}
  \int_{\tau}^{\tau+T}\mathcal{W}\big(w_{\ep}(t)\big)\dt\leq C_{\tau,T},\qquad\forall T>\ep,
 \end{equation}
 \begin{equation}
  \label{eq-aprtre}
  \int_0^T\|w_{\ep}''(s)\|_{\mathrm{W}'}^2\ds\leq C_{T}.
 \end{equation}
 \item[(c)] \textbf{Convergence}. Every sequence $w_{\ep_i}$ (with $\ep_i\downarrow0$) admits
 a subsequence which is convergent in the weak topology of $H^1_{\text{loc}}([0,+\infty);L^2)$
 to a function $w$ that satisfies \eqref{eq-ic} (where the latter condition
 is meant as an equality in $\mathrm{W}'$). In addition,
 \begin{equation}
  \label{eq-reg}
  w'\in L_{\text{loc}}^{\infty}([0,\infty);L^2)\qquad\text{and}\qquad w''\in L_{\text{loc}}^2([0,\infty);\mathrm{W}').
 \end{equation}
 \item[(d)] \textbf{Energy inequality}. Letting
 \[
  \mathcal{E}(t):=\frac{1}{2}\int_{\R^n}|w'(t,x)|^2\dx+\mathcal{W}\big(w(t)\big),
 \]
 there holds
 \begin{equation}
  \label{eq-enineq}
  \mathcal{E}(t)\leq\left(\sqrt{\mathcal{E}(0)}+\sqrt{\frac{t}{2}\int_0^t\int_{\R^n}|f(s,x)|^2\dx\ds\,}\right)^{2},\qquad\text{for a.e.} \quad t\geq0.
 \end{equation}
 \item[(e)] \textbf{Weak solution of \eqref{eq-wave}}. Assuming furthermore that, for some real numbers $m>0,\,\lambda_k\geq0,\,p_k>1$, $\mathcal{W}$ takes the form of
 \begin{equation}
  \label{eq-Wass}
  \mathcal{W}(v)=\frac{1}{2}\|v\|_{\dot{H}^m}^2+\sum_{0\leq k<m}\frac{\lambda_k}{p_k}\int_{\R^n}|\nabla^kv(x)|^{p_k}\dx,
 \end{equation}
 then the limit function $w$ satisfies
 \begin{equation}
  \label{eq-waveweak}
  \int_0^{\infty}\int_{\R^n}w'(t,x)\varphi'(t,x)\dx\dt = \int_0^{\infty}\left\langle\nabla\mathcal{W}\big(w(t)\big),\varphi(t)\right\rangle\dt-\int_0^{\infty}\int_{\R^n}f(t,x)\varphi(t,x)\dx\dt
 \end{equation}
 for every $\varphi\in C_0^{\infty}(\R^+\times\R^n)$, namely, solves \eqref{eq-wave} in the sense of distributions.
 \end{itemize}
\end{theorem}

\begin{remark}
 Note that the functional defined in \eqref{eq-Wass} satisfies Assumption \ref{ass-W} with $\mathrm{W}=\{v\in H^m:\nabla^kv\in L^{p_k},\,0\leq k<m\}$ (for details see \cite{ST2}). Recall also that $\|v\|_{\dot{H}^m}$ is the $L^2$ norm of $|\xi|^m\,\widehat{v}(\xi)$, where $\widehat{v}$ is the Fourier transform of $v$. The typical case is $m\in\mathbb{N}$ when $\|v\|_{\dot{H}^m}$ reduces to $\|\nabla^mv\|_{L^2}$.
\end{remark}

\begin{remark}
 Throughout, solutions of \eqref{eq-wave}-\eqref{eq-ic} obtained via Theorem \ref{result}, are called \emph{variational solutions}.
\end{remark}

\begin{remark}
 \label{rem-extensions}
 We mention that several variants of Theorem \ref{result} can be proved. For instance, one can introduce \emph{nonconstant coefficients} in \eqref{eq-Wass}, possibly exploiting some G$\mathring{\mathrm{a}}$rding type inequalities to keep $\mathcal{W}$ coercive. Also, one can consider more general lower-order terms (with proper convexity and growth assumptions) like \emph{powers of single partial derivatives}. In any case, the main point is that $\mathcal{W}$ be \emph{quadratic} (and coercive) in the highest order terms, namely that equation \eqref{eq-wave} be \emph{semilinear}. Furthermore, as pointed out in \cite{ST1,ST2}, the minimization approach can be adapted, without significative changes, to the case of a generic (sufficiently smooth) open set $\Omega\subset\R^n$ with Dirichlet or Neumann boundary conditions.
\end{remark}

It is worth stressing that $\mathcal{E}$ is formally preserved by variational solutions of \eqref{eq-wave}-\eqref{eq-ic}, in the sense that
\begin{equation}
 \label{eq-meccons}
 \mathcal{E}(t)=\mathcal{E}(0)+\int_0^t(f(s),w'(s))_{L^2}\ds,\qquad\forall t\geq0.
\end{equation}
However, we are not able to prove enough regularity for such solutions in order to solve the long-standing problem of the energy conservation for weak solutions of \eqref{eq-wave}. Anyway, a formal Gr\"onwall argument based on \eqref{eq-meccons} reveals that the energy estimate \eqref{eq-enineq} is ``close'' to being optimal.


\section{A preliminary tool: the average operator}
\label{sec-average}

The study of integrals with an exponential weight plays a central role in our investigation. Therefore, it is worth recalling the definition of \emph{average} operator, introduced in \cite{ST2}.

\begin{definition}
 The \emph{average operator} is the linear operator that associates any measurable function $h:[0,\infty]\to[0,\infty]$ with the function $\mathcal{A}h$, given by
 \[
  \mathcal{A}h\,(t):=\int_t^{\infty}e^{-(s-t)}\,h(s)\ds,\qquad t\geq0.
 \]
\end{definition}

\noindent We also recall that, as $\mathcal{A}h\,(0)<\infty$, $\mathcal{A}h$ is absolutely continuous on intervals $[0,T]$, for all $T>0$, and that
\[
 (\mathcal{A}h)'=\mathcal{A}h-h.
\]
In addition, one can iterate the action of $\mathcal{A}$, thus obtaining
\begin{equation}
 \label{eq-avter}
 \mathcal{A}^2h\,(t):=\mathcal{A}(\mathcal{A}h)\,(t)=\int_t^{\infty}e^{-(s-t)}\,(s-t)\,h(s)\ds
\end{equation}
(for details see \cite{ST2}). Finally, we stress that $\mathcal{A}h$ is well defined (and all the previous properties are valid) even when $h$ is a changing sign function, provided that it satisfies $\mathcal{A}|h|\,(0)<\infty$.

\medskip
Now, we show some relevant results that will be widely used in the sequel.

\begin{lemma}
 \label{lem-aveq}
 Let $h:[0,\infty)\to[0,\infty)$ be a function such that $\mathcal{A}h\,(0)<\infty$. Then, for every $\tau\geq0$ and every $\delta>0$
 \begin{equation}
  \label{eq-avequno}
  \int_{\tau}^{\tau+\delta}\mathcal{A}h\,(s)\ds=\int_{\tau}^{\tau+\delta}h(s)\ds+\mathcal{A}h\,(\tau+\delta)-\mathcal{A}h\,(\tau).
 \end{equation}
 If, in addition, $\mathcal{A}^2h\,(0)<\infty$, then
 \begin{equation}
  \label{eq-aveq}
  \int_{\tau}^{\tau+\delta}\mathcal{A}^2h\,(s)\ds = \int_{\tau}^{\tau+\delta}h(s)\ds+\mathcal{A}h\,(\tau+\delta)-\mathcal{A}h\,(\tau)+\mathcal{A}^2h\,(\tau+\delta)-\mathcal{A}^2h\,(\tau).
 \end{equation}
\end{lemma}

\begin{proof}
 By the Fubini theorem
 \begin{align*}
  \int_{\tau}^{\tau+\delta}\mathcal{A}\,h(s)\ds & =\int_0^{\infty}e^{-y}\,h(y)\left(\int_{\tau}^{\tau+\delta}\chi_{[0,y]}(s)\,e^s\ds\right)\,dy=\\[.2cm]
                                                & =\int_{\tau}^{\tau+\delta}e^{-y}\,h(y)\left(\int_{\tau}^{y}e^s\ds\right)\,dy+\int_{\tau+\delta}^{\infty}e^{-y}\,h(y)\left(\int_{\tau}^{\tau+\delta}e^s\ds\right)\,dy
 \end{align*}
 and then easy computations yield \eqref{eq-avequno}. Iterating the same argument one immediately finds \eqref{eq-aveq}.
\end{proof}

\begin{lemma}
 \label{lem-poinuno}
 For every $\alpha>1$ there exists a constant $C_{\alpha}>0$ such that for all $h\in H_{\text{loc}}^1([0,\infty);L^2)$
 \begin{equation}
  \label{eq-poinuno}
  \mathcal{A}\|h(\cdot)\|_{L^2}^2\,(t)\leq\alpha\|h(t)\|_{L^2}^2+C_{\alpha}\mathcal{A}\|h'(\cdot)\|_{L^2}^2\,(t),\qquad\forall t\geq0.
 \end{equation}
\end{lemma}

\begin{proof}
 Let $t\geq0$ and $a>t$. By assumption, for a.e. $x\in\R^n$ the function $h(\,\cdot\,,x)$ belongs to $H^1((t,a))$. Then, integrating by parts and using Cauchy-Schwarz,
 \begin{align*}
  \int_t^a e^{-s}\,|h(s,x)|^2\ds & \leq e^{-t}\,|h(t,x)|^2+2\int_t^a e^{-s}\,h(s,x)h'(s,x)\ds\\[.2cm]
                                 & \leq e^{-t}\,|h(t,x)|^2+2\left(\int_t^a e^{-s}\,|h(s,x)|^2\ds\right)^{1/2}\left(\int_t^a e^{-s}\,|h'(s,x)|^2\ds\right)^{1/2}.
  \end{align*}
 Since $2\sqrt{bc}\leq\nu b+\tfrac{1}{\nu}c$ for every $\nu>0$, we can split the last product and, for any choice of $\nu<1$, we find that
 \[
  \int_t^a e^{-s}\,|h(s,x)|^2\ds\leq\frac{1}{1-\nu}e^{-t}\,|h(t,x)|^2+\frac{1}{\nu(1-\nu)}\int_t^a e^{-s}\,|h'(s,x)|^2\ds.
 \]
 Now, integrating over $\R^n$ and letting $a\to\infty$, we obtain \eqref{eq-poinuno}, where $\alpha=\frac{1}{1-\nu}$ and $C_{\alpha}=\frac{1}{\nu(1-\nu)}=\frac{\alpha^2}{\alpha-1}$.
\end{proof}

\begin{lemma}
 \label{lem-poindue}
 For every $\beta>1$ there exists a constant $C_{\beta}>0$ such that for all $h\in H_{\text{loc}}^1([0,\infty);L^2)$
 \begin{equation}
  \label{eq-poindue}
  \mathcal{A}^2\|h(\cdot)\|_{L^2}^2\,(t)\leq\beta\|h(t)\|_{L^2}^2+C_{\beta}\left(\mathcal{A}\|h'(\cdot)\|_{L^2}^2\,(t)+\mathcal{A}^2\|h'(\cdot)\|_{L^2}^2\,(t)\right),\qquad\forall t\geq0.
 \end{equation}
\end{lemma}

\begin{proof}
 Let again $t\geq0$ and $a>t$. An easy change of variable yields
 \[
  \int_t^a e^{-(s-t)}(s-t)\,|h(s,x)|^2\ds=\int_0^{a-t}e^{-\tau}\tau\,|g(\tau,x)|^2\,d\tau
 \]
 with $g(\tau,x)=h(\tau+t,x)$. Then, arguing as in the proof of the previous lemma, we see that
 \begin{align*}
   \int_0^{a-t}e^{-\tau}\tau\,|g(\tau,x)|^2\,d\tau\leq & \; \int_0^{a-t}e^{-\tau}\,|g(\tau,x)|^2\,d\tau+\\[.2cm]
                                                       & \; +2\left(\int_0^{a-t}e^{-\tau}\tau\,|g(\tau,x)|^2\,d\tau\right)^{1/2}\left(\int_0^{a-t}e^{-\tau}\tau\,|g'(\tau,x)|^2\,d\tau\right)^{1/2}.
  \end{align*}
 Now, by Young inequality, for every $\nu\in(0,1)$
 \[
  \int_0^{a-t}e^{-\tau}\tau\,|g(\tau,x)|^2\,d\tau\leq\frac{1}{1-\nu}\int_0^{a-t}e^{-\tau}\,|g(\tau,x)|^2\,d\tau+\frac{1}{\nu(1-\nu)}\int_0^{a-t}e^{-\tau}\tau\,|g'(\tau,x)|^2\,d\tau.
 \]
 Hence, integrating over $\R^n$, changing the variables back and letting $a\to\infty$, we have
 \[
  \mathcal{A}^2\|h(\cdot)\|_{L^2}^2\,(t)\leq\alpha\mathcal{A}\|h(\cdot)\|_{L^2}^2\,(t)+C_{\alpha}\mathcal{A}^2\|h'(\cdot)\|_{L^2}^2\,(t)
 \]
 (where $\alpha=\tfrac{1}{1-\nu}$ and $C_{\alpha}=\tfrac{1}{\nu(1-\nu)}$). Finally, combining with \eqref{eq-poinuno} and setting $\beta=\alpha^2$ and $C_{\beta}=\alpha C_{\alpha}$, we obtain \eqref{eq-poindue}.
\end{proof}

\begin{remark}
 Setting $t=0$ in Lemma \ref{lem-poinuno} we recover Lemma \ref{lem-hmuprop}. In addition, note that we do not claim that any integral appearing in \eqref{eq-poinuno} and \eqref{eq-poindue} is necessarily finite.
\end{remark}


\section{Minimizers and first properties}
\label{min_prop}

The search of the minimizers mentioned in the previous sections is actually performed on an auxiliary functional. For a given a function $\phi:[0,\infty)\times\R^n\to\R$, define
\begin{equation}\label{defjeps}
 J_{\ep}(u):=H_{\ep}(u)-S(u),
\end{equation}
where
\[
 H_{\ep}(u):=\int_0^{\infty}e^{-t}\left(\frac{1}{2\ep^2}\int_{\R^n}|u''(t,x)|^2\dx+\mathcal{W}\big(u(t)\big)\right)\dt,
\]
and
\begin{equation}
 \label{eq-funzjsor}
 S(u):=\int_0^{\infty}\int_{\R^n}e^{-t}\,\phi(t,x)u(t,x)\dx\dt.
\end{equation}
One can see that $J_{\ep}$ is \emph{equivalent} to $F_{\ep}$ in the sense that, setting $\phi(t,x)=f_{\ep}(\ep t,x)$, there results $F_{\ep}(w)=\ep J_{\ep}(u)$, whenever $u$ and $w$ are related by the change of variable $u(t,x)=w(\ep t,x)$. Hence, properly scaling the boundary conditions (namely, as in \eqref{eq-bc}), the existence of minimizers $w_{\ep}$ for $F_{\ep}$ is equivalent to the existence of minimizers $u_{\ep}$ for $J_{\ep}$ and, in particular,
\[
 w_{\ep}(t,x)=u_{\ep}(t/\ep,x),\qquad t\geq0,\quad x\in\R^n.
\]
On the other hand, in contrast to $F_{\ep}$, $J_{\ep}$ presents integrals with a weight independent of $\ep$, thus simplifying the investigation.

\medskip
For functions $v=v(t,x)$, it is convenient to define the weighted $L^2$ ``norm''
\[
\Vert v\Vert_{\LL}^2:=\int_0^{\infty}\int_{\R^n}e^{-t}\,|v(t,x)|^2\dx\dt,
\]
with the proviso that we regard it as a functional (with values in $[0,+\infty]$)
rather than a norm proper.

\medskip
Throughout, for fixed $\eps$, we make the following assumptions on $\phi(t,x)$:
\begin{gather}
  \label{suppcomt} \phi(t,x)=0\qquad\forall t>T^*,\qquad\text{with}\quad \eps^2\, T^*\leq \sqrt{\ep},\\[.25cm]
\label{eq-assogr} \|\phi\|_\LL \leq \ep,\\
\label{assAq}\eps \int_0^t \mathcal{A}^2\|\phi(\cdot)\|_{L^2}^2\,(s)\,ds\leq
\gamma(\eps t+t_\eps)+\eps^2\quad\forall t\geq 0,
\end{gather}
where $t_\eps>0$ satisfies $\lim_{\ep\downarrow0} t_\eps=0$ while
\begin{equation}
\label{eq-gamma}
\gamma(t):=\int_0^t \Vert f(s)\Vert_{L^2}^2\,ds,\quad
t\geq 0,
\end{equation}
quantifies the growth in time of the forcing term $f\in L_{loc}^2([0,\infty);L^2)$ of \eqref{eq-wave}.

\begin{proposition}
 \label{prop-minimi}
 Let $w_0,w_1\in\mathrm{W}$ (with $\mathrm{W}$ defined by \eqref{eq-domW}) and $\ep\in(0,1)$. Then, under Assumption \ref{ass-W}, $J_{\ep}$ admits a minimizer $u_{\ep}$ in the class of functions $u\in H_{\text{loc}}^2([0,\infty);L^2)$ satisfying the boundary conditions
 \begin{equation}
  \label{eq-bc}
  u(0)=w_0,\qquad u'(0)=\ep w_1.
 \end{equation}
 Moreover,
 \begin{equation}
  \label{eq-levest}
  H_{\ep}(u_{\ep})\leq\mathcal{W}(w_0)+\ep C.
 \end{equation}
\end{proposition}

\noindent In order to prove Proposition \ref{prop-minimi}, we recall the following facts (for more see \cite[Lemma 2.3]{ST1}).

\begin{lemma}
 \label{lem-hmuprop}
 If $u\in H^2_{\text{loc}}([0,\infty);L^2)$, then
 \begin{equation}
  \label{eq-hmuno}
\Vert u'\Vert_\LL^2
\leq 2\,\Vert u'(0)\Vert_{L^2}^2
+4\,\Vert u''\Vert_\LL^2
 \end{equation}
 and
 \begin{align}
  \label{eq-hmdue}
\Vert u\Vert_\LL^2
\leq 2\,\Vert u(0)\Vert_{L^2}^2 +
8\,\Vert u'(0)\Vert_{L^2}^2
+16\,
\Vert u''\Vert_\LL^2.
 \end{align}
\end{lemma}

\begin{proof}[Proof of Proposition~\ref{prop-minimi}]
Lat $M$ be the set of functions in  $H^2_{\text{loc}}([0,\infty);L^2)$ satisfying
\eqref{eq-bc}. If $u\in M$, then $S(u)$ is finite by \eqref{suppcomt}, so that
$J_\eps(u)$ is well defined (possibly equal to $+\infty$).
If $J_\eps(u)$ is finite, then, since $\mathcal{W}\geq0$,
the finiteness of $H_\ep(u)$ implies
that the last integral
in \eqref{eq-hmdue}
is finite, and using Cauchy-Schwarz, \eqref{eq-assogr} and \eqref{eq-hmdue} we have
\[
|S(u)|\leq
\|\phi\|_\LL
\|u\|_\LL
\leq\eps C
\left(1+\|u''\|_\LL\right),
\]
where $C$ takes into account (via \eqref{eq-bc}) also the $L^2$ norms of $u(0)$ and $u'(0)$.
 Moreover,
from the definition of $J_\eps$ and last inequality we have
\begin{equation}\label{coerc}
   J_{\ep}(u)\geq
\frac{1}{2\ep^2}\Vert u''\Vert_\LL^2
+\int_0^\infty e^{-t}\mathcal{W}\bigl(u(t)\bigr)\,dt
-
\eps C
\left(1+\|u''\|_\LL\right),
\end{equation}
so that
$\|u''\|_\LL$ can be controlled in terms of
$J_\eps(u)$: using again \eqref{eq-hmuno} and \eqref{eq-hmdue}, we see that
$J_{\ep}$ is coercive in $M$ with respect to the topology of
$H_{\text{loc}}^2([0,\infty);L^2)$,
so that every
minimizing sequence
has a subsequence weakly convergent in $H_{\text{loc}}^2([0,\infty);L^2)$, which also preserves \eqref{eq-bc}.
The weak semicontinuity of $H_\eps(u)$ (building on Assumption~\ref{ass-W}) was proved
in \cite[proof of Lemma 3.1]{ST2}: since $S(u)$ is a weakly continuous functional, the existence
of a minimizer $u_\eps$ is established.

Now set $\psi(t,x):=w_0(x)+\ep t w_1(x)$, and observe
that $\psi\in M$ and $\psi''\equiv 0$. Moreover in
 \cite[proof of Lemma 3.1]{ST2}
it is proved that
 \[
 H_\eps(\psi)=
  \int_0^{\infty}e^{-t}\,\mathcal{W}\big(\psi(t)\big)\dt\leq\mathcal{W}(w_0)+C\ep,
 \]
 while by a direct computation, using Cauchy-Schwarz and \eqref{eq-assogr}, one has
 \[
  -S(\psi)\leq\left(\|w_0\|_{L^2}+\sqrt{2}\,\ep\|w_1\|_{L^2}\right)
  \|\phi\|_\LL
  \leq C\eps.
 \]
 Thus $J_\ep(\psi)\leq \mathcal{W}(w_0)+C\ep$, and
then also $J_\ep(u_\ep)\leq \mathcal{W}(w_0)+C\ep$
since $u_\ep$
is a minimizer. So, in particular,
$J_\ep(u_\ep)\leq C$: combining with \eqref{coerc} (written with $u=u_\eps$),
by Young's inequality one
can easily obtain $\|u_\ep''\|_\LL \leq \eps C$ as a byproduct. This, in turn,
can be plugged into \eqref{coerc} (with $u=u_\eps$) to estimate the last term,
thus finding
\begin{equation*}
   J_{\ep}(u_\ep)\geq
\frac{1}{2\ep^2}\Vert u''_\ep\Vert_\LL^2
+\int_0^\infty e^{-t}\mathcal{W}\bigl(u_\ep(t)\bigr)\,dt
-
\eps C.
\end{equation*}
Finally, \eqref{eq-levest} follows from the last inequality,
recalling that $J_\ep(u_\ep)\leq \mathcal{W}(w_0)+C\ep$.
 \end{proof}

\begin{remark}
 In the sequel we will always assume that $\ep\in(0,1)$, as in Proposition \ref{prop-minimi}.
\end{remark}

Now, we introduce some notation. Given a minimizer $u_{\ep}$ of $J_{\ep}$, we define
\begin{equation}
 \label{eq-WandD}
 \mathcal{W}_{\ep}(t):=\mathcal{W}\big(u_{\ep}(t)\big),\quad\forall t\geq0,\qquad D_{\ep}(t):=\displaystyle\frac{1}{2\ep^2}\|u_{\ep}''(t)\|_{L^2}^2,\quad\mbox{for a.e.  }t>0,
\end{equation}
and
\[
 L_{\ep}(t):=D_{\ep}(t)+\mathcal{W}_{\ep}(t).
\]
We also set
\[
 \Phi_{\ep}(t):=\big(\phi(t),u_{\ep}'(t)\big)_{L^2},\qquad\forall t\geq0,
\]
and define the \emph{kinetic energy} function as
\begin{equation}
 \label{eq-kinen}
 K_{\ep}(t):=\frac{1}{2\ep^2}\|u_{\ep}'(t)\|_{L^2}^2,\qquad\forall t\geq0.
\end{equation}
Note that $K_{\ep}$ is absolutely continuous on intervals $[0,T]$, with $T>0$, and that
\[
 K_{\ep}'(t)=\frac{1}{\ep^2}\big(u_{\ep}'(t),u_{\ep}''(t)\big)_{L^2},\qquad\mbox{for a.e. }t>0.
\]

\begin{proposition}
 \label{prop-jder}
 Let $w_0,\,w_1\text{ and }\mathcal{W}$ satisfy the assumptions of Proposition \ref{prop-minimi} and let $u_{\ep}$ be a minimizer of $J_{\ep}$. Then, for every $g\in C^2([0,\infty))$ constant for large $t$ and with $g(0)=0$,
 \begin{equation}
  \label{eq-relder}
  \begin{array}{l}
   \displaystyle\int_0^{\infty}e^{-t}\,\big(g'(t)-g(t)\big)L_{\ep}(t)\dt+\int_0^{\infty}e^{-t}\,g(t)\Phi_{\ep}(t)\dt+\\[.5cm]
   \hspace{5cm}\displaystyle-\int_0^{\infty}e^{-t}\,\big(4g'(t)D_{\ep}(t)+g''(t)K_{\ep}'(t)\big)\dt=g'(0)R(u_{\ep}),
  \end{array}
 \end{equation}
 where the linear functional
 \begin{equation}
  \label{eq-resto}
  R(u_{\ep}):=\ep\int_0^{\infty}e^{-t}\,t\,\left(-\left\langle\nabla\mathcal{W}\big(u_{\ep}(t)\big),w_1\right\rangle+(\phi(t),w_1)_{L^2}\right)\dt
 \end{equation}
 satisfies the estimate
  \begin{equation}
  \label{eq-restest}
  |R(u_{\ep})|\leq C\ep.
   \end{equation}
\end{proposition}

\begin{proof}
We proceed exactly as in \cite[proof of Proposition 4.4]{ST2}.
For small $\delta$,  we use the diffeomorphism  $\varphi_{\delta}(t):=t-\delta g(t)$ to define
   $U_{\delta}(t):=u_{\ep}\big(\varphi_{\delta}(t)\big)+t\ep\delta g'(0)w_1$,
    which is an admissible competitor of $u_\eps$ in the minimization
    of $J_\eps$, since it satisfies the initial conditions
    \eqref{eq-bc}. Then, since $u_\eps$ is a minimizer and
    $U_\delta=u_\eps$ when $\delta=0$, one has
\begin{equation}
\label{deriv}
  \left.\frac{\partial}{\partial\delta}J_{\ep}(U_{\delta})\right|_{\delta=0}=\left.\frac{\partial}{\partial\delta}H_{\ep}(U_{\delta})\right|_{\delta=0}-\left.\frac{\partial}{\partial\delta}S(U_{\delta})\right|_{\delta=0}=0,
\end{equation}
which (computing the derivatives) yields \eqref{eq-relder}. Indeed,
the derivative of $H_\eps(U_\delta)$ has been computed in
\cite[proof of Proposition 4.4]{ST2}, and it produces all the terms
in \eqref{eq-relder} except, of course, the integral of $\Phi_\ep$ and
the integrand involving $\phi(t)$ in \eqref{eq-resto}. On the other hand,
recalling \eqref{eq-funzjsor}, using \eqref{suppcomt}, \eqref{eq-assogr}
 and dominated convergence
one can check that
 \begin{align*}
  \left.\frac{\partial}{\partial\delta}S(U_{\delta})\right|_{\delta=0} & \, =
\int_0^\infty e^{-t}\int_{\R^n}\phi(t,x)\big(-g(t)u'_\eps(t,x)
+t\eps g'(0)w_1(t,x)\big)\,dx\,dt\\[.2cm]
& \, =
-
\int_0^\infty e^{-t} g(t)\Phi_\eps(t)\,dt+\eps g'(0)
\int_0^\infty e^{-t} t \,\,\big(\phi(t),w_1\big)_{L^2}\,dt,
\end{align*}
whence \eqref{deriv} reduces to \eqref{eq-relder}.

 Finally, combining \eqref{eq-diffW} and \eqref{eq-levest} as in \cite{ST2}, one has
 \begin{align*}
 \left|\int_0^{\infty}e^{-t}\,t\,\left\langle\nabla\mathcal{W}\big(u_{\ep}(t)\big),w_1
 \right\rangle\dt\right| & \leq \, C \bigg(1+\int_0^\infty e^{-t}\,t\mathcal{W}^{\theta}(u_\ep(t))\dt\bigg)\\[.3cm]
 & \leq \, C\big(1+H_\ep(u_\ep)\big)\,\leq\,C(1+\ep),
 \end{align*}
 while from Cauchy-Schwarz and \eqref{eq-assogr}
 \[
  \left|\int_0^{\infty}e^{-t}\,t\,\big(\phi(t),w_1\big)_{L^2}\dt
\right|
\leq \Vert w_1\Vert_{L^2}
\int_0^{\infty}e^{-t}\,t\,\Vert \phi(t)\Vert_{L^2}\dt
\leq
\Vert w_1\Vert_{L^2}
\left(\int_0^{\infty}e^{-t}\,t^2\,dt\right)^{\frac 1 2}
\Vert \phi\Vert_\LL\leq C\eps
 \]
 and hence inequality \eqref{eq-restest} is satisfied.
\end{proof}

\noindent This result has an immediate consequence.

\begin{corollary}
Using the notation of Section \ref{sec-average} for the operator $\mathcal{A}$, one has
 \begin{equation}
  \label{eq-relzero}
  \mathcal{A}^2L_{\ep}\,(0)+4\mathcal{A}D_{\ep}\,(0)-\mathcal{A}L_{\ep}\,(0)=\mathcal{A}^2\Phi_{\ep}\,(0)-R(u_{\ep})
 \end{equation}
 and
 \begin{equation}
  \label{eq-relt}
  \mathcal{A}^2L_{\ep}\,(t)+4\mathcal{A}D_{\ep}\,(t)-
\mathcal{A}L_{\ep}\,(t)=\mathcal{A}^2\Phi_{\ep}\,(t)-K_{\ep}'(t),
\qquad\mbox{for a.e. } t>0.
 \end{equation}
\end{corollary}

\begin{proof} Recalling \eqref{eq-avter},
\eqref{eq-relzero} is formally obtained choosing $g(t)=t$ in
\eqref{eq-relder}, but this goes beyond the assumptions of
Proposition~\ref{prop-jder}. However, as shown
in \cite[proof of Corollary 4.5]{ST2},
it suffices to approximate $g(t)=t$ from below,
 by suitable functions $g_k$ satisfying
the assumptions of
Proposition~\ref{prop-jder}, and pass to the limit in \eqref{eq-relder}.
Since one can arrange for $g'_k(0)=1$, only the integrals on the left
hand side of
\eqref{eq-relder} are actually involved, and the one with $\Phi_\ep$
(the only novelty with respect to \cite[Corollary 4.5]{ST2})
 passes to the limit by dominated convergence, using \eqref{suppcomt}
and \eqref{eq-assogr}.

Finally, also \eqref{eq-relt} is proved exactly as in
\cite[proof of Corollary 4.5]{ST2} (the only novelty being the term
with $\Phi_\eps$ that can be treated as described above), and we
omit the details. We just mention
that \eqref{eq-relt} (if written with $T$ in place of $t$)
is formally obtained choosing $g(t)=(t-T)^+$ in
\eqref{eq-relder}: then $g''(t)$ is a Dirac delta at $t=T$,
which produces the last term in \eqref{eq-relt}.
\end{proof}


\section{The approximate energy}
\label{sec-est_en}

Now we study the \emph{approximate energy}, a quantity that has been first introduced in \cite{ST2} and whose investigation is crucial for the proof of our main results.

\begin{definition}
 Let $u_{\ep}$ be a minimizer of $J_{\ep}$ obtained via Proposition \ref{prop-minimi}. The \emph{approximate energy} associated with $u_{\ep}$ is the function $E_{\ep}:[0,\infty)\to[0,\infty)$ defined by
 \begin{equation}
  \label{eq-appenesp}
  E_{\ep}(t):=\frac{1}{2\ep^2}\int_{\R^n}|u_{\ep}'(t,x)|^2\dx+\int_t^{\infty}e^{-(s-t)}(s-t)\,\mathcal{W}\big(u_{\ep}(s)\big)\ds.
 \end{equation}
\end{definition}

\begin{remark}
 Recalling \eqref{eq-avter}, \eqref{eq-WandD} and \eqref{eq-kinen}, \eqref{eq-appenesp} reads
 \[
  E_{\ep}(t)=K_{\ep}(t)+\mathcal{A}^2\mathcal{W}_{\ep}\,(t),\qquad t\geq0.
 \]
 In addition we stress that, in view of \eqref{defea}, $\mathcal{E}_\ep(t)=E_\ep(t/\ep)$.
\end{remark}

\medskip
The value of $E_{\ep}$ at $t=0$ can be estimated simply using \eqref{eq-relzero}.

\begin{lemma}[Estimate for $E_{\ep}(0)$]
 We have
 \begin{equation}
  \label{eq-appenzero}
  E_{\ep}(0)\leq\frac{1}{2}\|w_1\|_{L^2}^2+\mathcal{W}(w_0)+C\sqrt{\ep}.
 \end{equation}
\end{lemma}

\begin{proof}
 From \eqref{eq-bc}, $E_{\ep}(0)=\frac{1}{2}\|w_1\|_{L^2}^2+\mathcal{A}^2\mathcal{W}_{\ep}\,(0)$. Since
$\mathcal{A}^2\mathcal{W}_{\ep}\,(0)\leq \mathcal{A}^2L_{\ep}\,(0)$,
from \eqref{eq-relzero} we obtain
 \[
  \mathcal{A}^2\mathcal{W}_{\ep}\,(0)
\leq
\mathcal{A}^2\Phi_{\ep}\,(0)+\mathcal{A}L_{\ep}\,(0)-R(u_{\ep}).
 \]
 Now, as $\mathcal{A}L_{\ep}\,(0)=H_{\ep}(u_{\ep})$,
 combining the previous inequality with \eqref{eq-levest} and \eqref{eq-restest} yields
 \begin{equation}
  \label{eq-appzau}
  \mathcal{A}^2\mathcal{W}_{\ep}\,(0)\leq \mathcal{A}^2\Phi_{\ep}\,(0)+\mathcal{W}(w_0)+C\ep.
 \end{equation}
 Moreover, using first \eqref{suppcomt} and then \eqref{eq-assogr} we have
 \[
  |\mathcal{A}^2\Phi_{\ep}\,(0)|\leq
  T^* \int_0^\infty e^{-t} \left( |u'_\ep(t)|,\, |\phi(t)|\right)_{L^2} \,dt
\leq
T^*  \|\phi\|_\LL\,  \|u_{\ep}'\|_\LL\leq
C \frac{\|u_{\ep}'\|_\LL}{\sqrt{\ep}}.
 \]
Since from \eqref{eq-hmuno}, \eqref{eq-bc} and \eqref{eq-levest} we have
 \[
  \|u_{\ep}'\|_\LL^2\leq C\ep^2\,\|w_1\|_{L^2}^2+
  C\int_0^{\infty}e^{-t}\,D_{\ep}(t)\dt\leq C\ep^2\,(1+H_{\ep}(u_{\ep}))\leq C\ep^2,
 \]
 we find that $|\mathcal{A}^2\Phi_{\ep}\,(0)|\leq C\sqrt{\ep}$. Hence, plugging back into \eqref{eq-appzau}, \eqref{eq-appenzero} is proved.
\end{proof}

\noindent Furthermore, we establish an upper bound for the time evolution of the approximate energy.

\begin{proposition}[Approximate energy estimate]
 \label{prop-appest}
 For every $\beta>1$, there exists a constant $C_{\beta}>0$ such that 
for every $T\geq 0$
\begin{equation}
  \label{stimalocT}
   \sqrt{ E_{\ep}(T/\eps)}\leq
 \sqrt{E_{\ep}(0)}+
\left(\sqrt{\eps C_\beta}+\sqrt{T \beta/2}\right)
\sqrt{\gamma(T+t_\eps)+\eps^2}
\qquad\forall \eps\in (0,1).
\end{equation}
In particular, for every $T\geq 0$ there exists $C_T$  such that
\begin{equation}
  \label{stimalocT2}
   E_{\ep}(t/\eps)\leq
C_T
\qquad\forall \eps\in (0,1),\quad \forall t\in [0,T].
\end{equation}

\end{proposition}

\noindent In order to prove Proposition \ref{prop-appest}, we must previously compute the derivative of $E_{\ep}$.

\begin{lemma}
 The approximate energy $E_{\ep}$ is absolutely continuous on every interval $[0,T]$,
 and
 \begin{equation}
  \label{eq-appder}
  E_{\ep}'(t)=-3\mathcal{A}D_{\ep}\,(t)-\mathcal{A}^2D_{\ep}\,(t)+\mathcal{A}^2\Phi_{\ep}\,(t),\qquad\mbox{for a.e. }t\geq0.
 \end{equation}
\end{lemma}

\begin{proof}
Arguing as in \cite[proof of Theorem 4.8]{ST2} one can see that
 \[
  E_{\ep}'(t)=K_{\ep}'(t)-\mathcal{A}L_{\ep}\,(t)+\mathcal{A}D_{\ep}\,(t)+\mathcal{A}^2L_{\ep}\,(t)-\mathcal{A}^2D_{\ep}\,(t).
 \]
Then \eqref{eq-relt} can be used to eliminate $K_{\ep}'(t)$, and \eqref{eq-appder} follows.
\end{proof}

\noindent Now, it also is convenient to recall, without proof, a well-known variant of the Gr\"onwall's lemma (see e.g. \cite[Proposition 2.3.1]{cherrier}).

\begin{lemma}
 \label{lem-gron}
 Let $c:[a,b]\to\R$ be a positive, differentiable and nondecreasing function. Let also $u$ and $v$ be two nonnegative functions such that $u\in C^0([a,b])$ and $v\in L^1([a,b])$. If we assume that $c,\,u\mbox{ and }v$ satisfy
 \[
  u(t)\leq c^2(t)+2\int_a^tv(s)\sqrt{u(s)}\ds,\qquad\forall t\in[a,b],
 \]
 then there results
 \[
 \sqrt{ u(t)}\leq c(t)+\int_a^tv(s)\ds,\qquad\forall t\in[a,b].
 \]
\end{lemma}

\begin{proof}[Proof of Proposition~\ref{prop-appest}]
 First, recall that by definition
 \[
  \mathcal{A}^2\Phi_{\ep}\,(t)=\int_t^{\infty}e^{-(s-t)}(s-t)\,\big(\phi(s),u_{\ep}'(s)\big)_{L^2}\ds.
 \]
 Now, observing that $e^{-(s-t)}(s-t)$ is a probability kernel on $[t,\infty)$, \eqref{eq-appder} implies
 \begin{equation}
 \label{yyy}
  E_{\ep}'(t)\leq-3\mathcal{A}D_{\ep}\,(t)-\mathcal{A}^2D_{\ep}\,(t)+N_{\phi}(t)\left(\mathcal{A}^2\|u_{\ep}'(\cdot)\|_{L^2}^2\,(t)\right)^{1/2}
 \end{equation}
 where $N_{\phi}(t)=\left(\mathcal{A}^2\|\phi(\cdot)\|_{L^2}^2\,(t)\right)^{1/2}$.
 By Lemma \ref{lem-poindue}, applied with $h=u_\ep'$, for every $\beta>1$ there
 exists a constant $C_{\beta}>0$ such that
\[
\left(\mathcal{A}^2\|u_{\ep}'(\cdot)\|_{L^2}^2\,(t)\right)^{1/2}
\leq
\sqrt{\beta}\,\|u_{\ep}'(t)\|_{L^2}+
\sqrt{C_{\beta}}\,\left(\mathcal{A}\|u_{\ep}''(\cdot)\|_{L^2}^2\,(t) +
   \mathcal{A}^2\|u_{\ep}''(\cdot)\|_{L^2}^2\,(t)\right)^{1/2}.
\]
Since $\|u_{\ep}''(\cdot)\|_{L^2}^2\,(t)=2\eps^2 D_\ep(t)$,
multiplying by $N_{\phi}(t)$ and using Young's inequality we find
\[
N_{\phi}(t)
\left(\mathcal{A}^2\|u_{\ep}'(\cdot)\|_{L^2}^2\,(t)\right)^{1/2}
\leq
\sqrt{\beta}\,N_{\phi}(t)\|u_{\ep}'(t)\|_{L^2}+
C_{\beta}\eps^2 N_{\phi}(t)^2 + \left(3\mathcal{A} D_\eps(t) +
   \mathcal{A}^2 D_{\ep}(t)\right)
\]
(where $C_\beta$ has been possibly redefined). Plugging into \eqref{yyy}, we obtain
\[
  E_{\ep}'(t)\leq
\sqrt{\beta}\,N_{\phi}(t)\|u_{\ep}'(t)\|_{L^2}
+
C_{\beta}\eps^2 N_{\phi}(t)^2
\leq
\sqrt{2\beta}\,\eps N_{\phi}(t) \sqrt{E_\eps(t)}
+
C_{\beta}\eps^2 N_{\phi}(t)^2
\]
and then, integrating,
 \[
  E_{\ep}(t)\leq E_{\ep}(0)+C_{\beta}\ep^2\int_0^tN_{\phi}^2(s)\ds
  +\sqrt{2\beta}\,\ep\int_0^t N_{\phi}(s)\sqrt{E_{\ep}(s)}\ds.
 \]
 Now, setting
 \[
  u(t)=E_{\ep}(t),\qquad v(t)=\eps \sqrt{\beta/2} N_{\phi}(t),\qquad
  c(t)^2=E_{\ep}(0)+C_{\beta}\,\ep^2\int_0^tN_{\phi}^2(s)\ds,
 \]
 assumptions of Lemma \ref{lem-gron} are satisfied and thus for every $t\geq0$
 \[
  \sqrt{E_{\ep}(t)}\leq \left(E_{\ep}(0)+C_{\beta}\ep^2\int_0^tN_{\phi}^2(s)\ds\right)^{1/2}
  +\eps \sqrt{\beta/2}\int_0^t N_{\phi}(s)\ds.
 \]
 Therefore, 
 \[
 \sqrt{ E_{\ep}(t)}\leq
 \sqrt{E_{\ep}(0)}+
 \sqrt{C_{\beta}} \, \ep\sqrt{\int_0^tN_{\phi}^2(s)\ds}+
 \eps \sqrt{\beta/2}\int_0^t  N_{\phi}(s)\ds
 \]
and,
applying Cauchy-Schwarz in the last integral, we find
 \[
 \sqrt{ E_{\ep}(t)}\leq
 \sqrt{E_{\ep}(0)}+
 \sqrt{\eps}\left(\sqrt{C_\beta}+\sqrt{t\beta/2}\right)
\sqrt{\eps\int_0^tN_{\phi}^2(s)\ds}.
 \]
On the other hand, \eqref{assAq} gives
\[
\eps\int_0^tN_{\phi}^2(s)\ds
=\eps\int_0^t \mathcal{A}^2\|\phi(\cdot)\|_{L^2}^2\,(s)\,ds\leq
\gamma(\eps t+t_\eps)+\eps^2
\]
so that,  
setting $t=T/\eps$, we obtain \eqref{stimalocT}. Then \eqref{stimalocT2} is immediate,
since 
the right hand side of \eqref{stimalocT} is
increasing with respect to $T$; moreover,
$t_\ep\downarrow0$ (decreasingly), $\beta$ can be fixed (e.g. $\beta=2$) and
$E_\eps(0)\leq C$ by \eqref{eq-appenzero}.
\end{proof}


\section{Proof of Theorem \ref{result}: parts (a) and (b)}

Now, we can use the tools developed in the previous sections in order to to prove the first parts of Theorem \ref{result}.

\medskip
Preliminarily, we need a result on the approximation of functions in $L_{\text{loc}}^2([0,\infty);L^2)$.

\begin{lemma}
 \label{lem-fapp}
 For every function $f\in L_{\text{loc}}^2([0,\infty);L^2)$ there exists a sequence $(f_{\ep})\subset L_{\text{loc}}^2([0,\infty);L^2)$ satisfying the following properties:
 \begin{itemize}
  \item[(i)] as $\ep\downarrow0$, $f_{\ep}\to f$ in $L^2([0,T];L^2)$ and $\|f_{\ep}\|_{L^2([0,T];L^2)}\uparrow\|f\|_{L^2([0,T];L^2)}$, for every $T>0$;
  \item[(ii)] $\mathrm{supp}\{f_{\ep}\}\subset[t_{\ep},T_{\ep}]\times\R^n$, with $t_{\ep}>0$ and $T_{\ep}<\infty$;
  \item[(iii)] as $\ep\downarrow0$, $t_{\ep}\downarrow0$ and $T_{\ep}\uparrow\infty$, and moreover
   $\eps T_{\ep}\leq \sqrt{\ep}$, 
$e^{-t_{\ep}/\ep}\left(1+\frac{T_{\ep}}{\ep}\right)\leq \ep^3$;
  \item[(iv)] for every $\ep\in(0,1)$,
$ \displaystyle   \int_{t_{\ep}}^{T_{\ep}}\|f_{\ep}(t)\|_{L^2}^2\dt\leq 1/\ep$;
  \item[(v)] for every $\ep\in(0,1)$,
  $ \displaystyle\int_0^{\infty}e^{-t}\,\|f_{\ep}(\ep t)\|_{L^2}^2\dt \leq \ep^3$.
 \end{itemize}
\end{lemma}

\begin{proof}[Proof of Lemma~\ref{lem-fapp}]
Defining
 \[
  f_{\ep}(t,x)=\chi_{(t_{\ep},T_{\ep})}(t)\,f(t,x),
 \]
it is clear that (i) and (ii) are satisfied, as
soon as $t_\ep\downarrow0$ and $T_\eps\to+\infty$. We first construct $T_\eps$.
The function
 \[
 \Gamma:[0,\infty)\to[0,\infty),\qquad
   \Gamma(t):=\int_0^t\left(1+\|f(s)\|_{L^2}^2\right)\ds
 \]
is continuous, increasing and surjective, and therefore the same is true of its inverse
$\Gamma^{-1}$. Letting, for instance,
$T_{\ep}=\min\left\{\Gamma^{-1}(1/\ep),1/\sqrt{\ep}\right\}$, we have that
$T_\eps\to+\infty$ and that the first part of (iii) is satisfied, such as (iv), since
 \[
  \int_{t_{\ep}}^{T_{\ep}}\|f_{\ep}(s)\|_{L^2}^2\ds=
  \int_0^{T_{\ep}}\|f(s)\|_{L^2}^2\ds<\Gamma(T_{\ep})\leq1/\ep.
 \]
Finally, we see that
 \begin{align*}
 \int_0^{\infty}e^{-t}\,\|f_{\ep}(\ep t)\|_{L^2}^2\dt
 =\frac{1}{\ep}\int_0^{\infty} e^{-t/\ep}\,\|f_{\ep}(t)\|_{L^2}^2\dt
 =\frac{1}{\ep}\int_{t_{\ep}}^{T_{\ep}} e^{-t/\ep}\,\|f_\ep(t)\|_{L^2}^2\dt
 \leq \frac{e^{-t_{\ep}/\ep}}{\ep^2}
 \end{align*}
 having used (iv). Hence, to fulfill (v), it suffices to have $e^{-t_{\ep}/\ep}\leq\eps^5$
 for every $\eps\in (0,1)$, which is achieved
 choosing for instance $t_\eps=k\sqrt{\eps}$ with $k$ large enough. Finally,
 since $T_\eps\leq 1/\sqrt{\eps}$,
 the same choice can also guarantee the second inequality in (iii).
\end{proof}

\noindent The previous lemma has an important corollary.

\begin{corollary}
 \label{cor-fond}
 Let $f\in L_{loc}^2([0,\infty);L^2)$ and $(f_\ep)$ be a sequence obtained via Lemma \ref{lem-fapp}. If we fix $\ep\in(0,1)$, then the function 
 \begin{equation}
  \label{eq-assfi}
  \phi(t,x):=f_{\ep}(\ep t,x),\qquad t\geq0,\quad x\in\R^n,
 \end{equation}
 satisfies \eqref{suppcomt}--\eqref{assAq}.
\end{corollary}

\begin{proof}
 First, one can easily see that \eqref{suppcomt} and \eqref{eq-assogr} are direct consequences of properties (ii), (iii) and (v) of Lemma \ref{lem-fapp}.
 
  On the other hand, if one applies Lemma \ref{lem-aveq} with $\tau=0$, $\delta=t$ and $h(t)=\phi(t)=f_\eps(\eps t)$,
  then
\[
\int_0^t \mathcal{A}^2\|\phi(\cdot)\|_{L^2}^2\,(s)\,ds
\leq\int_0^t \|\phi(s)\|_{L^2}^2\ds
+\mathcal{A}\|\phi \|_{L^2}^2\,(t)+\mathcal{A}^2\|\phi\|_{L^2}^2\,(t).
 \]
 Now, from (ii) of Lemma \ref{lem-fapp} (with some changes of variable) we find that
 \begin{align*}
 \mathcal{A}\|\phi \|_{L^2}^2\,(t)+\mathcal{A}^2\|\phi\|_{L^2}^2\,(t)
 & = \, \int_0^{T^*}e^{-s}\left(1+s\right)\|\phi(s+t)\|_{L^2}^2\ds\\[.2cm]
 & = \, \eps^{-1} \int_0^{T_\eps}e^{-s/\eps}\left(1+\tfrac{s}{\eps}\right)\|f_\eps(s+\eps t)\|_{L^2}^2\ds.
 \end{align*}
 If we split the integral in two parts, then, from (iii) and (iv) in Lemma \ref{lem-fapp},
 \begin{align*}
\eps^{-1}  \int_{t_{\ep}}^{T_{\ep}}e^{-s/\ep}\,\left(1+\tfrac{s}{\ep}\right)\,\|f_{\ep}(s+\eps t)\|_{L^2}^2\ds & 
\leq \eps^{-1}
e^{-t_{\ep}/\ep}\,\left(1+\tfrac{T_{\ep}}{\ep}\right)\int_{t_{\ep}}^{T_{\ep}}\|f_{\ep}(s+\eps t)\|_{L^2}^2\ds
\leq \ep,
 \end{align*}
 while, recalling that $e^{-x}\,(1+x)\leq1$ for every $x\geq0$,
 \[
\eps^{-1}  \int_0^{t_{\ep}}e^{-s/\ep}\,\left(1+\tfrac{s}{\ep}\right)\,\|f_{\ep}(s+\eps t)\|_{L^2}^2\ds 
\leq \eps^{-1}\int_0^{t_{\ep}}\|f_{\ep}(s+\eps t)\|_{L^2}^2\ds=
\eps^{-1}\int_{\eps t}^{\eps t+t_{\ep}}\|f_{\ep}(s)\|_{L^2}^2\ds.
 \]
Therefore, recalling the definition of $\gamma$ given by \eqref{eq-gamma},
 \begin{equation*}
 \int_0^t \mathcal{A}^2\|\phi(\cdot)\|_{L^2}^2\,(s)\,ds
\leq\int_0^t \|\phi(s)\|_{L^2}^2\ds+
\eps^{-1}\int_{\eps t}^{\eps t+t_{\ep}}\|f_{\ep}(s)\|_{L^2}^2\ds
+\eps\leq \eps^{-1}\gamma(\eps t+t_\eps)+\eps
  \end{equation*}
and \eqref{assAq} follows.
\end{proof}

We can now prove the first part of Theorem \ref{result}.

\begin{proof}[Proof of Theorem~\ref{result}: part (a)]
 Let $(f_\ep)$ be a sequence obtained via Lemma \ref{lem-fapp}. If we set \eqref{eq-assfi} in \eqref{eq-funzjsor}, \eqref{suppcomt}--\eqref{assAq} and all the hypothesis of Proposition \ref{prop-minimi} are satisfied and hence we obtain a minimizer $u_{\ep}$, in the class of
 functions $u\in H_{\text{loc}}^2([0,\infty);L^2)$ subject to \eqref{eq-bc}, that fulfills \eqref{eq-levest}. Now, as
  \[
   F_{\ep}(w)=\ep J_{\ep}(u)\qquad\mbox{whenever}\qquad u(t,x)=w(\ep t,x),
  \]
  if $w_{\ep}$ is defined by
  \begin{equation}
  \label{eq-relmin}
  w_{\ep}(t,x)=u_{\ep}(t/\ep,x)\qquad t\geq0,\quad x\in\R^n,
 \end{equation}
  then it is the required minimizer.
\end{proof}

\begin{remark}
\label{rem-connections}
It is worth stressing that, under the assumptions of Theorem \ref{result} and \eqref{eq-assfi}, \eqref{eq-relmin} provides a direct connection between the minimizers of $J_\ep$ obtained via Proposition \ref{prop-minimi} and the minimizers of $F_\ep$.
Throughout, we will massively use this relation and, in particular, the fact that, setting \eqref{eq-assfi} with $(f_\ep)$ obtained via Lemma \ref{lem-fapp}, all the results proved in Sections \ref{min_prop}$\&$\ref{sec-est_en} are valid. We will also tacitly assume that the hypothesis of Theorem \ref{result} are satisfied.
\end{remark}

\noindent The proof of item (b) of Theorem \ref{result} requires two further auxiliary results.

\begin{lemma}[Euler-Lagrange equation of $u_{\ep}$]
 If $\eta(t,x)=\varphi(t)h(x)$, where $h\in\mathrm{W}$ and
$\varphi\in C^{1,1}([0,\infty))$ satisfies $\varphi(0)=\varphi'(0)=0$, then
 \begin{equation}
  \label{eq-elue}
  \frac{1}{\ep^2}\int_0^{\infty}e^{-t}\,\big(u_{\ep}''(t),\eta''(t)\big)_{L^2}\dt=\int_0^{\infty}e^{-t}\,\left(-\left\langle\nabla\mathcal{W}\big(u_{\ep}(t)\big),\eta(t)\right\rangle+\big(f_{\ep}(\ep t),\eta(t)\big)_{L^2}\right)\dt.
 \end{equation}
 Moreover, the same conclusion holds if $\eta\in C_0^{\infty}(\R^+\times\R^n)$.
\end{lemma}

\begin{proof}
When $\eta(t,x)=\varphi(t)h(x)$,
 \eqref{eq-elue} is obtained letting $g(\delta)=J_{\ep}(u_{\ep}+\delta\eta)$
and observing that $g'(0)=0$, since
$u_{\ep}$ is a minimizer of $J_{\ep}$ and
$u_{\ep}+\delta\eta$ is an admissible competitor.
The case where $\eta\in C_0^{\infty}(\R^+\times\R^n)$ follows by a density argument
(see \cite[proof of Lemma 5.1]{ST2} for more details).
The novelty here, with respect to
\cite{ST2},
 is just the term with $f_\eps$ in \eqref{eq-elue}, which originates from the additional
 term $S(u)$ in \eqref{defjeps}.
\end{proof}

\begin{lemma}[Representation formula for $u_{\ep}''$]
 For all $h\in\mathrm{W}$
 \begin{equation}
  \label{eq-rep}
  \frac{1}{\ep^2}\big(u_{\ep}''(\tau),h\big)_{L^2}
  =-\mathcal{A}^2\omega_1\,(\tau)+\mathcal{A}^2\omega_2\,(\tau),\qquad\mbox{for a.e.}\quad \tau>0,
 \end{equation}
 where $\omega_1(\tau)=\left\langle\nabla\mathcal{W}\big(u_{\ep}(\tau)\big),h\right\rangle$
 and $\omega_2(\tau)=\big(f_{\ep}(\ep \tau),h\big)_{L^2}$.
\end{lemma}

\begin{proof}
For every $h\in W$ and every $\tau>0$,
\eqref{eq-rep} formally
follows from \eqref{eq-elue} choosing $\eta(t,x)=\varphi(t)h(x)$, with $\varphi(t)=(t-\tau)^+$
so that $\varphi''(t)$ is a Dirac delta at $t=\tau$.
Indeed, \eqref{eq-rep}
can be proved rigorously (at every Lebesgue point $\tau$ of $\big(u_{\ep}''(\tau),h\big)_{L^2}$)
by approximating $\varphi(t)=(t-\tau)^+$ with $C^{1,1}$ functions, exactly
as in \cite[proof of (2.11)$\&$(2.16)]{ST2}.
\end{proof}

\begin{proof}[Proof of Theorem~\ref{result}: part (b)]
Note that, by \eqref{eq-relmin} and \eqref{eq-kinen}, $K_{\ep}(t/\ep)=\frac{1}{2}\|w_{\ep}'(t)\|_{L^2}^2$ and, since $K_{\ep}(t/\ep)\leq E_{\ep}(t/\ep)$, \eqref{stimalocT2} entails \eqref{eq-apruno}.

 On the other hand, arguing as in \cite[proof of Theorem 2.4]{ST2}, we see that
 \[
  \int_s^{s+1}\mathcal{W}_{\ep}(t)\dt\leq C\big(1+E_{\ep}(s-1)\chi_{[1,\infty)}(s)\big),\qquad\forall s\geq0,
 \]
 so that, setting $s=\tau/\ep$, (with some changes of variable)
 \[
  \int_{\tau}^{\tau+\ep}\mathcal{W}\big(w_{\ep}(t)\big)\dt\leq C\ep\big(1+E_{\ep}(\tau/\ep-1)\chi_{[\ep,\infty)}(\tau)\big)\leq C\ep\big(1+C_{\tau-\ep}\chi_{[\ep,\infty)}(\tau)\big),\qquad\forall \tau\geq0,
 \]
 where $C_{\tau-\ep}$ is the constant provided by \eqref{stimalocT2} (when $T=\tau-\ep$). As this constant is increasing with respect to time, one sees that for every $\tau\geq0$ and every $T>\ep$
 \[
  \int_{\tau}^{\tau+T}\mathcal{W}\big(w_{\ep}(t)\big)\dt\leq C\ep\sum_{i=1}^{\left[\frac{T}{\ep}\right]+1}(1+C_{\tau+(i-1)\ep})\leq C(1+C_{T+\tau+1})
 \]
 (where $\left[\tfrac{T}{\ep}\right]$ denotes the integer part of $\tfrac{T}{\ep}$), so that \eqref{eq-aprdue} is proved.

 Finally, we must prove \eqref{eq-aprtre}. By \eqref{eq-diffW}, one can see that $|\omega_1(t)|\leq C\|h\|_{\mathrm{W}}\big(1+\mathcal{W}_{\ep}(t)\big)$ and consequently
 \begin{equation}
  \label{eq-auno}
  |\mathcal{A}^2\omega_1\,(t)|\leq C\|h\|_{\mathrm{W}}\big(1+E_{\ep}(t)\big).
 \end{equation}
 On the other hand
 \[
  |\mathcal{A}^2\omega_2\,(t)|\leq\|h\|_{L^2}\int_t^{\infty}e^{-(s-t)}\,(s-t)\,\|f_{\ep}(\ep s)\|_{L^2}\ds
 \]
 and then, from \eqref{eq-dom_emb} and Jensen inequality, there results
 \begin{equation}
  \label{eq-adue}
  |\mathcal{A}^2\omega_2\,(t)|\leq C\|h\|_{\mathrm{W}}\left(\mathcal{A}^2\|f_{\ep}(\ep\,\cdot\,)\|_{L^2}^2\,(t)\right)^{1/2}.
 \end{equation}
 Combining \eqref{eq-auno} and \eqref{eq-adue} with \eqref{eq-rep}, we obtain that
 \[
  \frac{1}{\ep^2}|\big(u_{\ep}''(t),h\big)_{L^2}|\leq C\|h\|_{\mathrm{W}}\left(1+E_{\ep}(t)+\left(\mathcal{A}^2\|f_{\ep}(\ep\,\cdot\,)\|_{L^2}^2\,(t)\right)^{1/2}\right),\qquad\mbox{for a.e.}\quad t>0
 \]
 and hence, as \eqref{eq-dom_emb} entails $L^2\hookrightarrow\mathrm{W}'$, that
 \[
  \frac{1}{\ep^2}\|u_{\ep}''(t)\|_{\mathrm{W}'}\leq C\left(1+E_{\ep}(t)+\left(\mathcal{A}^2\|f_{\ep}(\ep\,\cdot\,)\|_{L^2}^2\,(t)\right)^{1/2}\right),\qquad\mbox{for a.e.}\quad t>0.
 \]
 Furthermore, in view of \eqref{eq-relmin} and \eqref{stimalocT2}, the last inequality reads
 \begin{equation}
  \label{eq-treaux}
  \|w_{\ep}''(t)\|_{\mathrm{W}'}\leq C\left(1+C_{t}+\left(\mathcal{A}^2\|f_{\ep}(\ep\,\cdot\,)\|_{L^2}^2\,(t/\ep)\right)^{1/2}\right),\qquad\mbox{for a.e.}\quad t>0.
 \end{equation}
 Now, recalling \eqref{assAq} and \eqref{eq-gamma}, in view of Corollary \ref{cor-fond}, one finds that for every $T>0$
 \[
  \int_0^T\mathcal{A}^2\|f_{\ep}(\ep\,\cdot\,)\|_{L^2}^2\,(t/\ep)\dt\leq \gamma(T+t_\ep)+\ep^2.
 \]
 Therefore, since $t_{\ep}\downarrow0$ when $\ep\downarrow0$, squaring and integrating inequality \eqref{eq-treaux} on $[0,T]$, we get \eqref{eq-aprtre}.
\end{proof}

\begin{remark}
 Due to the presence of the source term in \eqref{eq-wave}, the estimate that we establish on $(w_{\ep}'')$ is much ``weaker'' than the one obtained in \cite{ST2} in the homogeneous case. However, as we show below, this does not compromise the proof.
\end{remark}


\section{Proof of Theorem \ref{result}: parts (c), (d) and (e)}

Preliminarily, we stress that, throughout, we deal with a sequence of minimizers $w_{\ep_i}$ and we will \emph{tacitly} extract several subsequences. However, for ease of notation, we will denote by $w_{\ep}$ the original sequence, as well as the subsequences we extract. The same holds for all the other quantities depending on $\ep$.

\begin{proof}[Proof of Theorem~\ref{result}: part (c)]
 Let $T>0$. By \eqref{eq-apruno} and \eqref{eq-aprtre}, we see that
 \[
  \|w_{\ep}\|_{H^1([0,T];L^2)}\leq C_T,\qquad \|w_{\ep}'\|_{L^{\infty}([0,T];L^2)}\leq C_T,\qquad \|w_{\ep}'\|_{H^1([0,T];\mathrm{W}')}\leq C_T.
 \]
 Arguing as in \cite[proof of Theorem 2.4]{ST2}, this is sufficient to prove convergence in $H^1([0,T];L^2)$, \eqref{eq-ic} (with the latter meant as an equality in $\mathrm{W}'$) and \eqref{eq-reg}.
\end{proof}

\begin{proof}[Proof of Theorem~\ref{result}: part (d)]
 Observe that, letting $l(t):=\mathcal{W}_{\ep}(t)$ and $m(t):=E_{\ep}(t)-K_{\ep}(t)$ in \cite[Lemma 6.1]{ST2}, we obtain
 \[
  Y(\delta a)\int_{T+\delta a}^{T+a}\mathcal{W}_{\ep}(t)\dt+\int_T^{T+a}K_{\ep}(t)\dt\leq\int_T^{T+a}E_{\ep}(t)\dt,
 \]
 where $Y(z):=\int_0^ze^{-s}\,s\ds$. Replacing $a$ with $a/\ep$ and $T$ with $T/\ep$, with a change of variable, the previous inequality reads
 \[
  Y\left(\frac{\delta a}{\ep}\right)\int_{T+\delta a}^{T+a}\mathcal{W}\big(w_{\ep}(t)\big)\dt+\frac{1}{2}\int_T^{T+a}\|w_{\ep}'(t)\|_{L^2}^2(t)\dt\leq\int_T^{T+a}E_{\ep}(t/\ep)\dt.
 \]
 Hence, from \eqref{stimalocT}, we see that for an arbitrary $\beta>1$
 \[
  \begin{array}{l}
  \displaystyle Y\left(\frac{\delta a}{\ep}\right)\int_{T+\delta a}^{T+a}\mathcal{W}\big(w_{\ep}(t)\big)\dt+\frac{1}{2}\int_T^{T+a}\|w_{\ep}'(t)\|_{L^2}^2(t)\dt \\[.5cm]
  \hspace{5cm} \displaystyle \leq\int_T^{T+a}\left(\sqrt{E_\ep(0)}+\left(\sqrt{\ep C_\beta}+\sqrt{t\beta/2}\right)\sqrt{\gamma(t+t_\ep)+\ep^2}\right)\dt.
  \end{array}
 \]
 Now, when $\ep\downarrow0$, by definition $Y\left(\frac{\delta a}{\ep}\right)\to1$, whereas by \eqref{eq-appenzero} and \eqref{eq-gamma}
 \[
  \int_T^{T+a}\left(\sqrt{E_\ep(0)}+\left(\sqrt{\ep C_\beta}+\sqrt{t\beta/2}\right)\sqrt{\gamma(t+t_\ep)+\ep^2}\right)\dt\to\int_T^{T+a}\left(\sqrt{\mathcal{E}(0)}+\sqrt{t\gamma(t)\beta/2}\right)\dt.
 \]
 Consequently, arguing as in \cite[proof of Theorem 2.4]{ST2},
 \[
  \int_{T+\delta a}^{T+a}\mathcal{W}\big(w(t)\big)\dt+\frac{1}{2}\int_T^{T+a}\|w'(t)\|_{L^2}^2(t)\dt\leq\int_T^{T+a}\left(\sqrt{\mathcal{E}(0)}+\sqrt{t\gamma(t)\beta/2}\right)\dt
 \]
 and, letting $\delta\downarrow0$ and, subsequently, dividing by $a$ and letting $a\downarrow0$,  we obtain
 \[
  \mathcal{W}\big(w(T)\big)+\frac{1}{2}\|w'(T)\|_{L^2}^2\leq \sqrt{\mathcal{E}(0)}+\sqrt{T\gamma(T)\beta/2},\qquad\mbox{for a.e.}\quad T\geq0.
 \]
 Since the inequality is valid for every $\beta>1$, letting $\beta\downarrow1$, \eqref{eq-enineq} follows.
\end{proof}

Finally, before proving part (e) of Theorem \ref{result}, we claim the following result, which can be established directly by \eqref{eq-elue} (see \cite[Lemma 6.2]{ST2}).

\begin{lemma}
 Let $w_{\ep}$ be a minimizer of $F_{\ep}$. Then, for every function $\varphi\in C_0^{\infty}(\R^+\times\R^n)$, there results
 \begin{equation}
  \label{eq-elwe}
  \begin{array}{l}
   \displaystyle\int_0^{\infty}\big(w_{\ep}'(t),\ep^2\,\varphi'''(t)+2\ep\varphi''(t)+\varphi'(t)\big)_{L^2}\dt=\\[,5cm]
   \hspace{4.5cm}\displaystyle=\int_0^{\infty}\left\langle\nabla\mathcal{W}\big(w_{\ep}(t)\big),\varphi(t)\right\rangle\dt-\int_0^{\infty}\big(f_{\ep}(t),\varphi(t)\big)_{L^2}\dt.
  \end{array}
 \end{equation}
\end{lemma}

\begin{proof}[Proof of Theorem~\ref{result}: part (e)]
 The goal, now, is to prove that, as $\ep\downarrow0$, equation \eqref{eq-elwe} ``tends'' to \eqref{eq-waveweak}. Let $\varphi\in C_0^{\infty}(\R^+\times\R^n)$ and $w$ be the function obtained at point (c). We immediately see that
 \[
  \int_0^{\infty}\big(w_{\ep}'(t),\ep^2\,\varphi'''(t)+2\ep\varphi''(t)+\varphi'(t)\big)_{L^2}\dt\to\int_0^{\infty}\big(w'(t),\varphi'(t)\big)_{L^2}\dt.
 \]
 and, by construction, that
 \[
  \int_0^{\infty}\big(f_{\ep}(t),\varphi(t)\big)_{L^2}\dt\to\int_0^{\infty}\big(f(t),\varphi(t)\big)_{L^2}\dt.
 \]
 Hence the core of the proof is to show that
 \[
  \int_0^{\infty}\left\langle\nabla\mathcal{W}\big(w_{\ep}(t)\big),\varphi(t)\right\rangle\dt\to\int_0^{\infty}\left\langle\nabla\mathcal{W}\big(w(t)\big),\varphi(t)\right\rangle\dt.
 \]
 However, this follows by exploiting \eqref{eq-apruno}, \eqref{eq-aprdue}, \eqref{eq-Wass} and \cite[Theorem 5.1]{lions2} as in \cite{ST2}.
\end{proof}


\section{Examples}
\label{sec-examples}

For the sake of completeness we show some examples of second order nonhomogeneous hyperbolic equations that satisfy the assumptions of Theorem \ref{result}.

\bigskip
\noindent\textbf{1. Linear equations.} A \emph{linear} hyperbolic equation (with constant coefficients and without dissipative terms) can be written as
\begin{equation}
 \label{eq-linear}
 w''=-\sum_{j\in\mathcal{R}}(-1)^{|j|}\,\partial^{2j}w+f,
\end{equation}
where $\mathcal{R}\subset\mathbb{N}^n$ is a finite set of multi-indices and $\partial^{2j}$ denotes partial differentiation in space with respect to the multi-index $2j$. If we set
\[
 \mathcal{W}(v)=\frac{1}{2}\sum_{j\in\mathcal{R}}\int_{\R^n}|\partial^jv|^2\dx,
\]
then \eqref{eq-wave} reads like \eqref{eq-linear} and, letting $\mathrm{W}=\{v\in L^2:\partial^jv\in L^2,\,\forall j\in\mathcal{R}\}$ and $\theta=1/2$, assumptions of Theorem \ref{result} are satisfied (in view of Remark \ref{rem-extensions}). In particular, for suitable choices of $\mathcal{R}, \eqref{eq-linear}$ reads:
\[
 w''=\Delta w+f,\qquad w''=\Delta w-w+f,\qquad\text{or}\qquad w''=-\Delta^2 w+f
 \]
 and than all of these equations (namely, \emph{D'Alembert}, \emph{Klein-Gordon} and \emph{Plate/Bi-harmonic} wave equations, respectively) admit a variational solution.

\bigskip
\noindent\textbf{2. Defocusing NLW equation.} The \emph{defocusing NLW} equation reads
\begin{equation}
\label{dNLW}
 w''=\Delta w-|w|^{p-2}w+f\qquad(p>1).
\end{equation}
Here the proper choice of $\mathcal{W}$ is
\[
 \mathcal{W}(v)=\int_{\R^n}\left(\frac{1}{2}|\nabla v|^2+\frac{1}{p}|v|^p\right)\dx,
\]
with $\mathrm{W}=H^1\cap L^p$ and $\theta=1-1/\max\{2,p\}$, so that the assumptions of Theorem \ref{result} are satisfied.

\bigskip
\noindent\textbf{3. Sine-Gordon equation.} For the \emph{Sine-Gordon} equation
\[
 w''=\Delta w-\sin w+f
\]
the suitable definition of $\mathcal{W}$ is
\[
 \mathcal{W}(v)=\int_{\R^n}\left(\frac{1}{2}|\nabla v|^2+1-\cos v\right)\dx
\]
and, letting $\mathrm{W}=H^1$ and $\theta=1/2$, Theorem \ref{result} applies (again in view of Remark \ref{rem-extensions}).

\bigskip
\noindent\textbf{4. Quasilinear wave equations.} Two famous examples of \emph{quasilinear} hyperbolic equations are
\begin{equation}
 \label{eq-quasilinear}
 w''=\Delta_p w+f\qquad\mbox{and}\qquad w''=\Delta_p w-|w|^{q-2}w+f\qquad(p,q>1,\,p\neq2).
\end{equation}
For these equations, good choices of $\mathcal{W}$ are provided by
\[
 \mathcal{W}(v)=\frac{1}{p}\int_{\R^n}|\nabla v|^p\dx\qquad\mbox{and}\qquad\mathcal{W}(v)=\int_{\R^n}\left(\frac{1}{p}|\nabla v|^p+\frac{1}{q}|v|^q\right)\dx
\]
(respectively). Here, letting $\mathrm{W}=\{v\in L^2:\nabla v\in L^p\}$ with $\theta=1-1/p$ in the former case and $\mathrm{W}=\{v\in L^2:\nabla v\in L^p,\,v\in L^q\}$ with $\theta=1-1/\max\{p,q\}$ in the latter case, Theorem \ref{result} holds up to item (e). It is an open problem then to establish the existence of a variational solution for both \eqref{eq-quasilinear}$_1$-\eqref{eq-ic} and \eqref{eq-quasilinear}$_2$-\eqref{eq-ic}.

\bigskip
\noindent\textbf{5. Higher order nonlinear equations.} A famous example of higher order hyperbolic equation is the \emph{nonlinear vibrating-beam} equation
\[
 w''=-\Delta^2w+\Delta_pw-|w|^{q-2}\,w+f\qquad(p,\,q>1)
\]
(see e.g. \cite{pausader,peletier}). The suitable choice of $\mathcal{W}$ here is
\[
 \mathcal{W}(v)=\int_{\R^n}\left(\frac{1}{2}|\Delta v|^2+\frac{1}{p}|\nabla v|^p+\frac{1}{q}|v|^q\right)\dx
\]
and, setting $\mathrm{W}=\{v\in H^2:\nabla v\in L^p,v\in L^q\}$ with $\theta=1-1/\max\{2,p,q\}$, Theorem \ref{result} holds.

\bigskip
\noindent\textbf{6. Kirchhoff equations.} These are typical examples of \emph{nonlocal} problems. For instance, consider the equation
\begin{equation}\label{kirch}
 w''=\left(\int_{\R^n}|\nabla w|^2\dx\right)\Delta w+f.
\end{equation}
The natural choice of $\mathcal{W}$ is given by
\[
 \mathcal{W}(v)=\frac{1}{4}\left(\int_{\R^n}|\nabla v|^2\dx\right)^2,
\]
and, setting $\mathrm{W}=H^1$ and $\theta=3/4$, Theorem \ref{result} applies except for part (e), which consequently remains an open problem.

\bigskip
\noindent\textbf{7. Wave equations with fractional Laplacian.} Further examples of nonlocal problems are provided by hyperbolic equations involving the \emph{fractional Laplacian}, as for instance
\[
 w''=-(-\Delta)^s-\lambda|w|^{p-2}\,w+f\qquad(0<s<1,\,\lambda\geq0,\,p>1).
\]
Here, if one takes the functional
\[
 \mathcal{W}(v)=c_{n,s}\int_{\R^n\times\R^n}\frac{|v(x)-v(y)|^2}{|x-y|^{n+2s}}\dx\,dy+\frac{\lambda}{p}\int_{\R^n}|v|^p\dx,
\]
which is the natural energy associated to the fractional Laplacian, then setting $\mathrm{W}=H^s\cap L^p$ and $\theta=1-1/\max\{2,p\}$  (when $\lambda>0$, or $\mathrm{W}=H^s$ and $\theta=1/2$  when $\lambda=0$) one sees that the assumptions of Theorem \ref{result} are satisfied.


\end{document}